\documentclass{amsart}

\usepackage{amssymb}
\usepackage{amsmath}

\numberwithin{equation}{section}

\newtheorem{theorem}{Theorem}[section]
\newtheorem{lemma}[theorem]{Lemma}

\newtheorem{prop}[theorem]{Proposition}
\newtheorem{remark}[theorem]{Remark}

\newtheorem{define}{Definition}[section]

\newcommand\pd{{\partial}}

\newcommand{\R}{{\mathbb{R}}}

\begin{document}

\title[Prandtl equations in three space variables]
{Global existence of weak solutions to the three-dimensional Prandtl equations with A special structure}

\author{C.-J. Liu}
\address{Cheng-Jie Liu
\newline\indent
Department of Mathematics, Shanghai Jiao Tong University,
Shanghai, China, and
\newline\indent
Department of Mathematics, City University of Hong Kong,
Hong Kong, China}
\email{cjliusjtu@gmail.com}

\author{Y.-G. Wang}
\address{Ya-Guang Wang
\newline\indent
Department of Mathematics, MOE-LSC and SHL-MAC, Shanghai Jiao Tong University,
Shanghai, China}
\email{ygwang@sjtu.edu.cn}

\author{T. Yang}
\address{Tong Yang
\newline\indent
Department of mathematics, City University of Hong Kong,
Hong Kong, China, and
\newline\indent
Department of Mathematics, Shanghai Jiao Tong University, Shanghai, China}
\email{matyang@cityu.edu.hk}




\date{}


\begin{abstract}
The global existence of weak solutions to the three space dimensional
 Prandtl equations is studied under some constraint on its structure.
This
is a continuation of our recent study on the local existence of classical
solutions with the same structure condition.   It reveals
 the sufficiency of the monotonicity condition on one
component of the tangential velocity field and the favorable condition on
pressure in the same direction that leads to global existence of weak solutions.
This generalizes
the result obtained by Xin-Zhang on the two-dimensional Prandtl equations to the three-dimensional setting.
\end{abstract}

\maketitle

\tableofcontents

\section{Introduction}

Consider the initial boundary value problem for the Prandtl
 boundary layer equations in three space variables,
\begin{equation}\label{pr}
\begin{cases}
\partial_t u+(u\partial_x+v\partial_y+w\partial_z) u+\partial_x p=\partial_z^2u,\quad &in~Q,\\
\partial_t v+(u\partial_x+v\partial_y+w\partial_z) v+\partial_y p=\partial_z^2v,\quad &in~Q,\\
\partial_x u+\partial_y v+\partial_z w=0,\quad &in~Q,\\
(u,v,w)|_{z=0}=0, \quad \lim\limits_{z\to+\infty}(u,v)=(U(t,x,y), V(t,x,y)),\\
(u,v)|_{t=0}= \big(u_0(x,y,z), v_0(x,y,z)\big),
\end{cases}\end{equation}
where
\(Q=\{t>0, (x,y)\in D, z>0\}\) with a fixed $D\subset \R^2$,
$(U(t,x,y), V(t,x,y))$ and $p(t,x,y)$ are the traces of the tangential velocity field and the pressure of the Euler
flow on the boundary $\{z=0\}$. Note that the traces satisfy
\begin{equation}\label{euler}
\begin{cases}
  \partial_t U+U\partial_x U+V\partial_y U+\partial_x p=0,\\
  \partial_t V+U\partial_xV+V\partial_yV+\partial_y p=0.
  \end{cases}
\end{equation}

Despite of its importance in physics, there are very few mathematical results on the Prandtl equations in three space variables. In fact, except the recent work \cite{lwy}
about the classical solution with special structure and those in the analytic framework \cite{Samm, CS}, most of the mathemtical studies on this foundamental system in boundary layer theory are limited to the problem in two space dimensions, cf. \cite{AWXY, e-1, Masmoudi-Wong, OA1, Ole, xin, XZ} and the references therein.

Recently in \cite{lwy}, we obtain the local well-posedness of classical solutions to the problem \eqref{pr} under some constraint on the structure of the solution, in order to avoid the appearance of secondary flow (\cite{moore}) in boundary layers. Precisely, assuming that for the Euler flow given in \eqref{euler}, $U(t,x,y)>0$, in the class of boundary layers that the direction of tangential velocity field is invariant in the normal variable $z$, and the $x-$component of velocity $u(t,x,y,z)$ is strictly increasing in $z$, $\partial_zu>0$, in \cite{lwy} we have constructed a classical solution to the problem \eqref{pr}, and it is linearly stable with respect to any three dimensional perturbation. In the class of this special structure, the solution of \eqref{pr} takes the form:
\begin{equation}\label{str}
\left(u(t, x,y,z), k(x,y)u(t,x,y,z), w(t,x,y,z)\right),\end{equation}
where the function $k(x,y)$ satisfies the following condition (H):
\begin{enumerate}
\item[(H1)]   the function $k$ depending only on $(x,y)$  satisfies the inviscid Burgers equation  in $D$,
\begin{equation}\label{eqkk}
k_x +kk_y=0;
\end{equation}

\item[(H2)] the outer Euler flow
$$(U(t,x,y), k(x,y)U(t,x,y), 0, p(t,x,y))$$
with $U(t,x,y)>0$, satisfies that from the system \eqref{euler},
\begin{equation*}
\begin{cases}
\pd_tU+U\pd_xU+kU\pd_yU+\pd_x p=0,\\
\pd_y p-k\pd_x p=0.
\end{cases}\end{equation*}
\end{enumerate}

Moreover, the authors recently observed in \cite{lwy-1} that for the shear flow $(u^s(t,z), v^s(t,z), 0)$ of the three-dimensional Prandtl equations, the special solution structure \eqref{str} is the only stable case.

Under the above assumption \eqref{str} of special solution structure,  the original problem \eqref{pr} of three dimensional Prandtl equations is reduced to the following one for two unknown functions $(u,w)$:
\begin{equation}\label{eqk2}\begin{cases}
\partial_t u+(u\partial_x +ku\partial_y +w
\partial_z) u-\partial_z^2 u=-\pd_x p,\quad &in~Q,\\[2mm]
\partial_x u+\partial_y (ku)+\partial_z w=0,\quad &in~Q,\\[2mm]
u|_{z=0}=w|_{z=0}=0, \quad \lim\limits_{z\rightarrow +\infty}u=U(t,x,y),\\[2mm]
u|_{\pd Q_-}=u_1(t,x,y,z),\quad u|_{t=0}=u_0(x,y,z),
\end{cases}\end{equation}
where $\pd Q_-=(0,\infty)\times\gamma_-\times\mathbb{R}_+$,
with
$$\gamma_-~=~\{(x,y)\in\pd D|~ (1,k(x,y))\cdot\vec{n}(x,y)<0\},$$
and $\vec{n}(x,y)$ is the unit outward normal vector of $D$ at $(x,y)\in\pd D$.

For this reduced problem under the assumption that
\begin{equation}\label{ass_ib}
\pd_z u_0>0,~\pd_z u_1>0,~\quad {\rm for}~z\geq0,
\end{equation}
 in the class of $\partial_zu>0$, we apply the method developed by Oleinik \cite{Ole} for two dimensional Prandtl equations. Precisely, by the Crocco transformation,
\begin{equation*}
\xi=x,~\eta=y,~\zeta=\frac{u(t,x,y,z)}{U(t,x,y)},~W(t,\xi,\eta,\zeta)=\frac{\pd_z u(t,x,y,z)}{U(t,x,y)},
\end{equation*}
 the problem \eqref{eqk2} becomes the following initial boundary value problem,
\begin{equation}\label{pr_main}
\begin{cases}
L(W)~\triangleq~\partial_t W+\zeta U(\partial_\xi +k\partial_\eta)W+A\pd_\zeta W+BW-W^2\pd_\zeta^2 W=0, \quad in~\Omega_T,\\
W|_{\zeta=1}=0,\quad W\pd_\zeta W|_{\zeta=0}=\frac{ p_x}{U},\\
W|_{\Gamma_-}=W_1(t,\xi,\eta,\zeta)\triangleq \frac{\pd_z u_1}{U},\\
W|_{t=0}=W_0(\xi,\eta,\zeta)\triangleq \frac{\pd_z u_0}{U},
\end{cases}\end{equation}
where
$$\Omega_T=\{(t,\xi,\eta,\zeta)|~0< t<T,~(\xi,\eta)\in D,~0\leq\zeta<1\},$$
$$\Gamma_-=\{(t,\xi,\eta,\zeta)|~0< t<T,~(\xi,\eta)\in\gamma_-,~0\leq\zeta<1\},$$
and
\begin{equation}\label{def_not}
A=-\zeta(1-\zeta)\frac{ U_t}{ U}-(1-\zeta^2) \frac{p_x}{U},\quad
~B=\frac{ U_t}{ U}+\zeta(U_x+kU_y)-\pd_y k\cdot \zeta U.
\end{equation}
For the problem \eqref{pr_main} of the degenerate parabolic equation, we established local existence
of classical solutions in \cite{lwy}.

As a continuation of the paper \cite{lwy},
the purpose of this paper is to prove the global in time existence of a weak solution to the problem \eqref{pr_main} for data
satisfying \eqref{ass_ib} and the   favourable pressure condition:
\begin{equation}\label{press}
p_\xi (t,\xi,\eta)~\leq~0,\quad {\rm for}~t>0,~(\xi,\eta)\in D.
\end{equation}
For this, we will adopt the approach introduced by Xin-Zhang in \cite{XZ} for the the two-dimensional Prandtl equations to the three-dimensional setting. As observed in \cite{prandtl}, the main motivation of introducing
the favourate condition on pressure is to avoid the separation of boundary layers.

For completeness,  the definition of weak solutions to \eqref{pr_main} is
given as follows.

\begin{define}
 A function $W(t,\xi,\eta,\zeta)\in L^\infty\Big(0,T;BV(\Omega)\Big)$ for some $T>0$
is called a weak solution of the problem \eqref{pr_main}
in $t<T,$ if the following conditions hold:

{\bf{(i)}} There exists a positive constant $C$ such that
\[C^{-1}(1-\zeta)~\leq~W(t,\xi,\eta,\zeta)~\leq~C(1-\zeta),\quad \forall (t,\xi,\eta,\zeta)\in \Omega_T.\]

{\bf{(ii)}} $W$ satisfies the first equation and the initial boundary conditions of \eqref{pr_main} in the  weak sense:
 \begin{equation*}
 \begin{split}
&\quad\int_0^T\int_\Omega\Big\{{1\over W}\Big[\psi_t+\zeta (U\psi)_\xi+\zeta (Uk\psi)_\eta+(A\psi)_\zeta-B\psi\Big]-W\psi_{\zeta\zeta}\Big\}d\xi d\eta d\zeta dt\\
&=\int_\Omega{1\over W_0}\psi|_{t=0}d\xi d\eta d\zeta-\int_0^T\int_D{p_x\over U}\cdot{\psi\over W}|_{\zeta=0}d\xi d\eta dt
+\int_0^T\int_{\gamma_-}\int_0^1{\zeta U\psi\over W_1}\cdot k_nd\zeta dl dt,
\end{split}\end{equation*}
for any test function
\(\psi(\tau,\xi,\eta,\zeta)\in C^\infty(\overline{\Omega_T})\)
satisfying
\[\psi=0,~{\rm at}~ t=T~{\rm or}~(\xi,\eta)\in\gamma_+;\quad
\psi_\zeta=0~~{\rm at} ~\zeta=0.\]
Here,
$$\gamma_+~=~\{(x,y)\in\pd D|~ (1,k(x,y))\cdot\vec{n}(x,y)>0\},$$
and the function
$k_n=(1,k)\cdot\vec{n}$ with $\vec{n}$ being the unit outward normal vector of $\gamma_-$.
\end{define}

The main result of this paper can be stated as follows.

\begin{theorem}\label{thm_global}
For the problem \eqref{pr_main}, and any  $T>0$, assume that  $k\in C^2(D)$, $U\in C^2((0,T)\times D)$, $p_x\in C^1((0,T)\times D)$ satisfy \eqref{press}, and the initial boundary data $W_0\in C^1(\Omega),W_1\in C^3(\Gamma_-)$ satisfy
\begin{equation}\label{ass_ib_glo}
C_0^{-1}(1-\zeta)\leq W_0,W_1\leq C_0(1-\zeta),
 \end{equation}
for a positive constant $C_0$. Then, there exists a weak solution $W(t,\xi,\eta,\zeta)\in L^\infty\Big(0,T;BV(\Omega)\Big)$ to the problem \eqref{pr_main} in the sense of Definition 1.1.
\end{theorem}



\section{The proof of the main result}

Following the approach introduced in \cite{XZ},  a viscous splitting method is used to construct a sequence of approximate solutions to the problem
\eqref{pr_main}.
 Precisely, divide the time interval $[0,T]$ into $n$ equal sub-intervals:
\[0=t_0<t_1<\cdots<t_{n-1}<t_n=T,\quad t_{i+1}-t_i={T\over n},~~{\rm for}~0\leq i\leq n-1.\]
First, in the time step $[0,t_1]$ or $[t_i, t_{i+1}]$ for an even $i$, we construct the approximate solution
 by solving the following initial boundary value problem for a porous media-type equation:
\begin{equation}\label{pr_porous}\begin{cases}
{1\over2}\pd_t W-W^2\pd_\zeta^2 W+A\pd_\zeta W+b W=0,\quad in~(0,t_1]\times\Omega~{\rm or}~(t_i,t_{i+1}]\times\Omega,\\
W|_{t=0}=W_0,~{\rm or}~W|_{t=t_i}=W(t_i,\xi,\eta,\zeta)~({\rm given~in~the~previous~step}),\\
 W|_{\zeta=1}=0,\quad WW_\zeta|_{\zeta=0}={p_x\over U};
\end{cases}\end{equation}
and in the time step $[t_i,t_{i+1}]$ for an odd $i$, we construct the approximate solution by solving the following  problem for a transport equation:
\begin{equation}\label{pr_trans}
\begin{cases}
{1\over2}\pd_t W+\zeta U(\pd_\xi+k\pd_\eta)W+(B-b) W=0, \quad in~(t_i,t_{i+1}]\times\Omega,\\
W|_{t=t_i}=W(t_i,\xi,\eta,\zeta)~({\rm given~in~the~previous~step}),\\
W|_{\Gamma_-}=W_1.
\end{cases}\end{equation}
Here, the coefficient function $b$ in \eqref{pr_porous} will be chosen later to satisfy the boundary condition $W|_{\Gamma_-}=W_1$ for the solution of \eqref{pr_porous}.
What is needed  is to prove that the function $W$ constructed in the  time interval $[0,T]$ is uniformly bounded in $n$ and
has a uniform total variation with respect to the spatial variables $\xi,\eta$ and $\zeta$. This implies that  as $n\rightarrow\infty$, the limit function of the approximate solutions $W$ constructed in \eqref{pr_porous}-\eqref{pr_trans} is a weak solution to the problem \eqref{pr_main}.
The proof is divided into the following subsections.

\subsection{Porous medium-type equation}

In this subsection, consider the following problem for a
porous medium-type equation
\begin{equation}\label{pr_por}
\begin{cases}
\pd_t W-W^2\pd_\zeta^2 W+A\pd_\zeta W+bW=0,\quad {\rm for} ~0<t<T,~(\xi,\eta,\zeta)\in\Omega,\\
W|_{t=0}=W_0(\xi,\eta,\zeta),\\
W\pd_\zeta W|_{\zeta=0}={p_x\over U}\leq0,\quad W|_{\zeta=1}=0.
\end{cases}\end{equation}

In order to match the boundary condition $W|_{\Gamma_-}=W_1$,  by observing that $W_1>0$ for $0\leq\zeta<1$ and
\[ W_1=\mathcal{O}(1-\zeta),\quad as~\zeta\rightarrow1,\]
we set
\begin{equation}\label{def_b}
b(t,\xi,\eta,\zeta)=-{f(\xi,\eta)\over W_1}\Big(\pd_t W_1-(W_1)^2\pd_\zeta^2 W_1+A|_{\Gamma_-}\cdot\pd_\zeta W_1\Big),
\end{equation}
where $f(\xi,\eta)$ is a non-negative smooth function defined on the closure of the domain $D$ satisfying  $f(\xi,\eta)|_{\gamma_-}=1$.
By the formulations in \eqref{def_not}:
\[A=-(1-\zeta)\Big[\zeta{U_t\over U}+(1+\zeta){p_x\over U}\Big],\]
and the assumption given in Theorem \ref{thm_global},  there exists a positive constant $M_0$ depending on the parameters of \eqref{pr_main}, such that for the function $b(t,\xi,\eta,\zeta)$ given in \eqref{def_b},
\begin{equation}\label{est_b}
\|b(t,\xi,\eta,\zeta)\|_{W^{1,\infty}(\Omega_T)}\leq M_0.
\end{equation}
The problem \eqref{pr_por} can be viewed as a one space dimensional problem by regarding variables $\xi$ and $\eta$ as parameters.

Note that the equation in \eqref{pr_por} is degenerate on the boundary $\{\zeta=1\}$. As \cite{XZ}, consider  the following uniformly approximated
 parabolic problem:
\begin{equation}\label{pr_por1}
\begin{cases}
\pd_t W_\epsilon-(W_\epsilon^2+\epsilon)\pd_\zeta^2 W_\epsilon+A\pd_\zeta W_\epsilon+bW_\epsilon=0,\quad for ~t>0,~(\xi,\eta,\zeta)\in\Omega,\\
W_\epsilon|_{t=0}=W_0>0,\\
W_\epsilon\pd_\zeta W_\epsilon|_{\zeta=0}={p_x\over U}\leq0,\quad W_\epsilon|_{\zeta=1}=0,
\end{cases}\end{equation}
for a positive constant $\epsilon>0$. It is  known that problem \eqref{pr_por1} has a unique smooth solution. After getting
some uniform bounds of $W_\epsilon$ with respect to $\epsilon$, we
can obtain a solution to the problem \eqref{pr_por} by taking $\epsilon\rightarrow0$
in \eqref{pr_por1}. In fact, we have

\begin{theorem}\label{thm_por}
Under the assumption of Theorem \ref{thm_global}, the problem \eqref{pr_por} has a unique solution $W\in BV(0,T;\Omega)$ for any fixed $T>0$. Moreover, $W$ has the following properties:
\\
(1) there exists a positive constant $\beta$, depending on $\|W_0\|_{L^\infty}$ and the $L^\infty$-norm of the parameters of \eqref{pr_por}, such that
      \begin{equation}\label{est_por1}
      \theta_0e^{-\beta t}\varphi\leq W(t,\xi,\eta,\zeta)\leq C_1e^{M_1t}(1-\zeta),
      \end{equation}
where
\begin{equation}\label{def_para}
\begin{split}
&\varphi=e^{{\pi\over2}\zeta}\sin{\pi\over2}(1-\zeta),\quad \theta_0=\min\{W_0/\varphi\},\\
&C_1=\max\Big\{\|{W_0\over1-\zeta}\|_{L^\infty},~\sqrt{\|{p_x\over U}\|_{L^\infty}}\Big\},
\quad
M_1=\|{A\over1-\zeta}-b\|_{L^\infty};
\end{split}\end{equation}

 (2) there exists a positive constant $C_2$, depending on $\|W\|_{L^\infty},\|\pd_\zeta W_{0}\|_{L^\infty}$ and the $C^1$-norm of the parameters of \eqref{pr_por}, such that
    \begin{equation}\label{est_porz}
    |W_\zeta|\leq C_2;\end{equation}

(3)  for any $t>0$,
 \begin{equation}\label{est_w-z}
\int_0^1|W_\zeta(t,\xi,\eta,\zeta)|d\zeta\leq\int_0^1| W_{0,\zeta}(\xi,\eta,\zeta)|d\zeta
+W(t,\xi,\eta,0)-W_0(\xi,\eta,0),
\end{equation}
and
\begin{equation}\label{est_w-x}
\begin{split}
\int_0^1{|W_\xi(t)|\over W^2(t)}(1-\zeta)^2d\zeta\leq &\int_0^1{|W_{0,\xi}|\over W_0^2}(1-\zeta)^2d\zeta+C_3t\Big(1+\int_0^1|W_{0,\zeta}|d\zeta\Big)\\
&+C_3\int_0^t\int_0^1
{|W_\xi(s)|\over W^2(s)}(1-\zeta)^2d\zeta ds.
\end{split}\end{equation}
Here, the positive constant $C_3$ depends on $\|W\|_{L^\infty}$ and the $C^1$-norm of the parameters of \eqref{pr_por}. Also, $W_\eta$ and $W_t$ satisfy similar estimates as \eqref{est_w-x} by simply replacing the partial derivative
with respect to $\xi$ in \eqref{est_w-x} by
the partial derivatives with respect to $\eta$ and $t$, respectively.
\end{theorem}

\begin{proof}[\bf{Proof.}]
To prove the inequality \eqref{est_por1},
we first show $W|_{\zeta=0}\geq0$ by using the favorable
condition on the pressure function.  For this, we can assume that ${p_x\over U}<0$ because we can replace it by ${p_x\over U}-\delta$ for some constant $\delta>0$ and then let $\delta\rightarrow0$.
Then, we want to show  that
\begin{equation}\label{est_bd}
W_\epsilon|_{\zeta=0}>0,\quad \forall \epsilon>0
\end{equation}
holds for the problem \eqref{pr_por1}. Otherwise, by using the continuity of $W_\epsilon$ and $W_0|_{\zeta=0}>0$, there exist  $\epsilon_0>0$ and a point $P$ on $\{\zeta=0\}$, such that $W_{\epsilon_0}|_{P}=0$. That is,  $W_{\epsilon_0}^2$ attains its minimum at $P$, which implies that
$\pd_\zeta W_{\epsilon_0}^2|_P~\geq~0$. But from the boundary condition of \eqref{pr_por1} on $\zeta=0$, we have
\[\pd_\zeta W_{\epsilon_0}^2|_{\zeta=0}=2W_{\epsilon_0}\pd_\zeta W_{\epsilon_0}|_{\zeta=0}=2{p_x\over U}<0,\]
which is a contradiction. Hence, we obtain \eqref{est_bd}, which implies that $W|_{\zeta=0}~\geq~0$ for the problem \eqref{pr_por} by letting $\epsilon\rightarrow0$.

Now, combining $W|_{\zeta=0}~\geq~0$ with the boundary condition $WW_\zeta|_{\zeta=0}={p_x\over U}\leq0$, we obtain $W_\zeta|_{\zeta=0}\leq0$. The rest of proof for \eqref{est_por1} is similar to that of Lemma 4.2 in \cite{XZ} so that we
omit the details. Moreover, there is a constant $m>0$, such that
\begin{equation}\label{est_porz0}
W|_{\zeta=0}\geq m,\quad W_\epsilon|_{\zeta=0}\geq m,\quad \forall \epsilon>0
\end{equation}
hold.

(2) We now turn to the estimate \eqref{est_porz}. Consider the
 problem \eqref{pr_por1} for $W_\epsilon$, and
the corresponding problem for $\pd_\zeta W_\epsilon$ as follows,
\begin{equation}\label{pr_por_z}
\begin{cases}L^\epsilon(\pd_\zeta W_\epsilon)-2W_\epsilon \pd_\zeta W_\epsilon \pd_{\zeta}^2W_\epsilon+(A_\zeta+b)\pd_\zeta W_\epsilon=-b_\zeta W_\epsilon,\\
\pd_\zeta W_\epsilon|_{t=0}=W_{0,\zeta},\\
W_\epsilon\pd_\zeta W_\epsilon|_{\zeta=0}={p_x\over U},
\quad \pd_\zeta^2 W_{\epsilon}|_{\zeta=1}=0,
\end{cases}\end{equation}
where the operator
\[L^\epsilon=\pd_t-(W_\epsilon^2+\epsilon)\cdot\pd_\zeta^2+A\cdot\pd_\zeta.\]
Here, note that
$\pd_{\zeta}^2W_\epsilon|_{\zeta=1}=0$ because $W_\epsilon|_{\zeta=1}=0$ and $A|_{\zeta=1}=0$ in the first equation in \eqref{pr_por1}.

Set $V=\pd_\zeta W_\epsilon-\alpha \zeta$ with $\alpha$ being a constant to be chosen later. It follows that
\begin{equation}\label{V-es}
\begin{split}
&L^\epsilon(V)-2W_\epsilon(V+\alpha\zeta)V_\zeta+(A_\zeta+b-2\alpha W_\epsilon)V\\
&=(2\alpha^2\zeta-b_\zeta)W_\epsilon-\alpha A-\alpha\zeta(A_\zeta+b)
\triangleq Y.\end{split}\end{equation}
From the first step,
 $Y$ is bounded provided that $\alpha$ is bounded.
Let $V_1=e^{-\beta t}V$ with $\beta>0$ being sufficiently large  satisfying
\[\beta+A_\zeta+b-2\alpha W_\epsilon\geq1, ~or~\beta\ge \|A_\zeta+b-2\alpha W_\epsilon\|_{L^\infty}+1 .\]
From \eqref{V-es} and \eqref{pr_por_z}, we have
\begin{equation}\label{eq_w-z}
L^\epsilon(V_1)-2W_\epsilon(V+\alpha\zeta)(V_1)_\zeta+(\beta+A_\zeta+b-2\alpha W_\epsilon)V_1=e^{-\beta t}Y,
\end{equation}
with the following initial and boundary conditions
\begin{equation}\label{ib_w-z}
V_1|_{t=0}=W_{0,\zeta}-\alpha\zeta,\quad W_\epsilon V_1|_{\zeta=0}=e^{-\beta t}{p_x\over U},
\quad \pd_\zeta V_1|_{\zeta=1}=-\alpha e^{-\beta t}.
\end{equation}

Firstly, note that for an arbitrarily fixed constant $\alpha>0$, from \eqref{ib_w-z}, we have
$\pd_\zeta V_1|_{\zeta=1}<0$. Also, from \eqref{est_porz0} and the relation in \eqref{ib_w-z}
 \[W_\epsilon V_1|_{\zeta=0}=e^{-\beta t}{p_x\over U}\leq0,
 \]
it implies that $V_1$ does not attain its positive maximum on $\zeta=1$ or $\zeta=0$.
Then, if $V_1$ attains its positive maximum in the interior of $\Omega_T$ or
when $t=T$, we have from \eqref{eq_w-z} that
\[V_1\leq\max\{e^{-\beta t}Y\}\leq \|Y\|_{L^\infty}.\]
If $V_1$ achieves its positive maximum when $t=0$, it follows that
\[V_1\leq\max\Big\{W_{0,\zeta}-\alpha\zeta
\Big\}\leq 
\|W_{0,\zeta}\|_{L^\infty}.
\]
Therefore, we conclude that
\(V_1\leq
\max\Big\{ \|Y\|_{L^\infty}, \|W_{0,\zeta}\|_{L^\infty}
\Big\},\)
which implies that
\begin{equation}\label{est_w-z3}
\pd_\zeta W_\epsilon\leq \alpha+ e^{\beta t}\max\Big\{\|Y\|_{L^\infty},\|W_{0,\zeta}\|_{L^\infty}
\Big\}.\end{equation}

Secondly, for an arbitrarily fixed  constant $\alpha<0$,
by considering the possible negative minimal points of $V_1$ on $\overline{\Omega}_T$, similar to the above argument, we have
\[V_1\geq -\max\Big\{\|Y\|_{L^\infty}, \|W_{0,\zeta}\|_{L^\infty},\|{p_x\over U}{1\over W_\epsilon}\Big|_{\zeta=0}\|_{L^\infty}\Big\}\]
which implies that
\begin{equation}\label{est_w-z4}
\pd_\zeta W_\epsilon\geq\alpha -e^{\beta t}\max\Big\{\|Y\|_{L^\infty},\|W_{0,\zeta}\|_{L^\infty},\|{p_x\over U}{1\over W_\epsilon}\Big|_{\zeta=0}\|_{L^\infty}\Big\}.
\end{equation}

Hence, combining \eqref{est_w-z3} with \eqref{est_w-z4} and letting $\alpha\rightarrow0$  yield that
\[|\pd_\zeta W_\epsilon|\leq e^{\beta t}\max\Big\{\|Y\|_{L^\infty},\|W_{0,\zeta}\|_{L^\infty},\|{p_x\over U}{1\over W_\epsilon}\Big|_{\zeta=0}\|_{L^\infty}\Big\},\]
where
$\beta\geq\|A_\zeta+b\|_{L^\infty}+1$ and $Y=-b_\zeta W_\epsilon.$
Thus, we obtain  \eqref{est_porz} as $\epsilon\rightarrow0$.

(3) The proofs of \eqref{est_w-z} and \eqref{est_w-x} are similar to those given
in Lemmas 4.6 and 4.7 of \cite{XZ}, respectively.
And the proof for the uniqueness of solution to the problem \eqref{pr_por} is similar to that of Theorem 4.1 in \cite{XZ}.
Thus, we omit the detail for brevity and this
completes the proof of the theorem.
\end{proof}

\subsection{Transport equation}

In this section, we will study the problem of transport equation \eqref{pr_trans} for $t\in (t_{i}, t_{i+1}]$ with $i$ being odd, that is, we consider the problem,
\begin{equation}\label{pr_tran}
\begin{cases}
{1\over2}\pd_t W+\zeta U(\pd_\xi+k\pd_\eta)W+b_1 W=0, \quad in~(t_i,t_{i+1}]\times\Omega,\\
W|_{t=t_i}=W(t_i,\xi,\eta,\zeta),\\
W|_{\Gamma_-}=W_1,
\end{cases}\end{equation}
where the function
\begin{equation}\label{def_b1}
b_1(t,\xi,\eta,\zeta)=(B-b)(t,\xi,\eta,\zeta)
\end{equation}
with functions $B$ and $b$ being given in \eqref{def_not} and \eqref{def_b}, respectively.

For any fixed $(t,\xi,\eta,\zeta)\in(t_{i},t_{i+1}]\times\Omega$,
the characteristics of the equation \eqref{pr_tran} passing through this point is denoted by:
\begin{equation*}
\Big(s,~\gamma_1(s;t,\xi,\eta,\zeta),~\gamma_2(s;t,\xi,\eta,\zeta),~\zeta\Big), \qquad s\in(t_{i},t_{i+1}]
\end{equation*}
with $\gamma_1$ and $\gamma_2$ being determined by
\begin{equation}\label{char}\begin{cases}
\gamma_1'(s;t,\xi,\eta,\zeta)~=~2\zeta\cdot U\Big(s,\gamma_1(s;t,\xi,\eta,\zeta),\gamma_2(s;t,\xi,\eta,\zeta)\Big),\\
\gamma_2'(s;t,\xi,\eta,\zeta)~=~2\zeta\cdot (kU)\Big(s,\gamma_1(s;t,\xi,\eta,\zeta),\gamma_2(s;t,\xi,\eta,\zeta)\Big),\\
(\gamma_1,\gamma_2)(t;t,\xi,\eta,\zeta)=(\xi,\eta).
\end{cases}\end{equation}
For simplicity of notations, in the following we will also use abbreviations $\gamma_1(s)$ and $\gamma_2(s)$ to represent $\gamma_1(s;t,\xi,\eta,\zeta)$ and $\gamma_2(s;t,\xi,\eta,\zeta)$ respectively, when without confusion.

Combining \eqref{char} with the property $k_\xi+k k_\eta=0$, we have
\[{d\over ds}k\Big(\gamma_1(s),\gamma_2(s)\Big)=0,\]
which implies that
\begin{equation}\label{char_k}
k\Big(\gamma_1(s),\gamma_2(s)\Big)
\equiv k(\xi,\eta), \quad \forall s\in(t_{i},t_{i+1}].
\end{equation}
Then, from \eqref{char} and \eqref{char_k} it follows
\begin{equation}\label{exp_char}
\gamma_2(s)=k(\xi,\eta)\cdot
\gamma_1(s)+\eta-k(\xi,\eta)\xi,\end{equation}
and
$\gamma_1(s)$ is given by
\begin{equation}\label{char1}
\begin{cases}
\gamma_1'(s)=2\zeta \cdot U\Big(s,\gamma_1(s),k(\xi,\eta)\cdot
\gamma_1(s)+\eta-k(\xi,\eta)\xi\Big),\\
\gamma_1(t)=\xi.
\end{cases}\end{equation}
Observe that $\gamma_1'(s)\geq0$, and the projection of this characteristic on the $(\xi, \eta)$-plane is a straight line passing through $(\xi,\eta)$ with slope $k(\xi,\eta)$. Moreover, the function $k(\xi,\eta)$ remains constant along this line.

Note that the solution of the above problem \eqref{char1} exists and is unique when the function
 $U(t, \xi, \eta)$ is Lipschitz in $(\xi, \eta)$.
The solution $W$ of the problem \eqref{pr_tran} is represented by
\begin{equation}\label{exp_tr}\begin{split}
W(t,\xi,\eta,\zeta)~=~&W\Big(s,\gamma_1(s),
\gamma_2(s),\zeta\Big)\exp\Big\{-\int_s^t
b_1\Big(\tilde s,\gamma_1(\tilde s),
\gamma_2(\tilde s),\zeta\Big)d\tilde s\Big\}.
\end{split}\end{equation}
Denote by
\begin{equation}\label{def_t}
t^*(t,\xi,\eta,\zeta)=\inf\Big\{\tilde t\in[t_{i},t]: \forall s\in(\tilde t,t], ~\Big(\gamma_1(s;t,\xi,\eta,\zeta),\gamma_2(s;t,\xi,\eta,\zeta)\Big)\in D\Big\}.
\end{equation}
Note that if $t^*>t_{i}$, then \[\Big(\gamma_1(t^*;t,\xi,\eta,\zeta),\gamma_2(t^*;t,\xi,\eta,\zeta)\Big)\in \pd D.\]

Denote by
\[\begin{split}Q_1=\Big\{&(t,\xi,\eta,\zeta)\in(t_{i},t_{i+1}]\times\Omega:
t^*(t,\xi,\eta,\zeta)=t_{i}\Big\},\end{split}\]
and
\[\begin{split}Q_2=\Big\{&(t,\xi,\eta,\zeta)\in(t_{i},t_{i+1}]\times\Omega:
~t^*(t,\xi,\eta,\zeta)>t_{i},~\Big(\gamma_1(t^*),\gamma_2(t^*)\Big)
\in \overline \gamma_-\Big\},\end{split}\]
where $\overline \gamma_-$ is the closure of $\gamma_-=\{(\xi,\eta)\in\pd D:~\left(1,k(\xi,\eta)\right)\cdot\vec{n}(\xi,\eta)<0\}
$ on the boundary $\pd D$.

To study the estimate of the solution to problem \eqref{pr_tran}, we first give the following proposition for the representation of the solution.

\begin{prop}\label{prop_tr}
For the problem \eqref{pr_tran}, we have
\[(t_{i},t_{i+1}]\times\Omega=Q_1\cup Q_2,\]
and at any $(t, \xi, \eta, \zeta)\in (t_{i},t_{i+1}]\times\Omega$, the solution of \eqref{pr_tran} can be represented by
\begin{equation}\label{def_tr}\begin{split}
&W(t,\xi,\eta,\zeta)\\
&=\begin{cases}W\Big(t_{i-1},\gamma_1(t_{i-1}),
\gamma_2(t_{i-1}),\zeta\Big)\cdot\exp\Big\{-\int_{t_{i-1}}^t
b_1\big(s,\gamma_1(s),\gamma_2(s),\zeta\big)ds\Big\},\quad in~Q_1,\\
W_1\Big(t^*,\gamma_1(t^*),
\gamma_2(t^*),\zeta\Big)\cdot\exp\Big\{-\int_{t^*}^t
b_1\big(s,\gamma_1(s),\gamma_2(s),\zeta\big)ds\Big\},\qquad\qquad in~Q_2.\\
\end{cases}\end{split}\end{equation}
\end{prop}

\begin{proof}[\bf Proof.]
It suffices to show that $(t_{i},t_{i+1}]\times\Omega=Q_1\cup Q_2,$
because the formulation \eqref{def_tr} follows immediately from \eqref{exp_tr}.
We divide the proof into the following two parts.

(1) Firstly, we claim that
\begin{equation*}
(t,\xi,\eta,0)\in Q_1,\quad {\rm for ~ all}~t\in(t_{i},t_{i+1}],~(\xi,\eta)\in D.
\end{equation*}
Indeed, from \eqref{exp_char} and \eqref{char1}
it follows that $\gamma_1'(s)=0$, and $(\gamma_1(s),\gamma_2(s))\equiv(\xi,\eta)$,
which implies that $t^*(t,\xi,\eta,0)=t_{i}$ by using the definition \eqref{def_t}.

(2) Next, we will prove that if
  $$(t,\xi,\eta,\zeta)\in\big((t_{i-1},t_i]\times\Omega\big)\setminus Q_1,\qquad\zeta>0,$$
then $(t,\xi,\eta,\zeta)\in Q_2$. Indeed, for such point $(t,\xi,\eta,\zeta)$, there is $t^*(t,\xi,\eta,\zeta)>t_{i}$ such that
\[(\xi^*,\eta^*)\triangleq\Big(\gamma_1({t^*}),\gamma_2(t^*)\Big)\in\pd D.\]
We then need to show that $P=(\xi^*,\eta^*)\in\overline \gamma_-$.

From \eqref{exp_char} and \eqref{char1}, we have
\begin{equation}\label{exp}\gamma_1'(s)>0,\quad \gamma_1(s)>\xi^*,~\forall s\in(t^*,t],\end{equation}
and
\begin{equation}\label{exp1}
\gamma_2(s)=k(\xi^*,\eta^*)\big(\gamma_1(s)-\xi^*\big)+\eta^*,\quad \forall s\in(t^*,t].\end{equation}

Without loss of generality, assume that in a neighborhood $P_\delta$ of the point $P=(\xi^*,\eta^*)$ in the $(\xi, \eta)$-plane, the boundary of $D$ is represented by a smooth function $\eta=f(\xi)$, so that $\eta>f(\xi)$ when $(\xi,\eta)\in D\cap P_\delta$. Then, the outward normal vector at the point $\Big(\xi,f(\xi)\Big)$ is given by
\begin{equation}\label{def_n}
\vec{n}\big(\xi,f(\xi)\big)=\frac{1}{\sqrt{1+\big(f'(\xi)\big)^2}}\Big(f'(\xi),-1\Big).
\end{equation}
By the definition of $t^*$, we have there exists a $\epsilon_0>0$ such that
\[\gamma_2(s)> f(\gamma_1(s)),\quad \forall s\in(t^*,t^*+\epsilon_0].\]
From \eqref{exp} and \eqref{exp1}, we know that there exists a constant $\epsilon_1>0$ such that
\begin{equation}\label{est_tr}
k(\xi^*,\eta^*)(\xi-\xi^*)+\eta^*> f(\xi),\quad \forall \xi\in(\xi^*,\xi^*+\epsilon_1].
\end{equation}

If $P$ does not belong to $\overline \gamma_-$, then there exists a $\delta_1<\delta$ such that
\[\big(1,k(\xi,\eta)\big)\cdot\vec{n}(\xi,\eta)\geq0,\quad \forall (\xi,\eta)\in\pd D\cap P_{\delta_1}\]
which implies from \eqref{def_n} that there exists a $\epsilon_2>0$, satisfying
\[\Big(1,k\big(\xi,f(\xi)\big)\Big)\cdot(f'(\xi),-1)\geq0,\quad \forall \xi\in(\xi^*-\epsilon_2,\xi^*+\epsilon_2).\]
This is,
\begin{equation}\label{der1}
f'(\xi)\geq k(\xi,f(\xi)),\quad \forall \xi\in(\xi^*-\epsilon_2,\xi^*+\epsilon_2).\end{equation}
For a fixed $\xi_0\in(\xi^*,\xi^*+\epsilon_3)$ with $\epsilon_3\leq\min\{\epsilon_1,\epsilon_2\}$, consider the function
\[F(\xi)=f(\xi)-f(\xi_0)-k\big(\xi,f(\xi)\big)(\xi-\xi_0),\quad \xi\in[\xi^*,\xi_0].\]
Since $k_\xi+kk_\eta=0$ from \eqref{eqkk}, it follows that
\[\begin{split}F'(\xi)&=f'(\xi)-k\big(\xi,f(\xi)\big)-\Big[k_\xi\big(\xi,f(\xi)\big)+f'(\xi)k_\eta\big(\xi,f(\xi)\big)\Big]\cdot(\xi-\xi_0)\\
&=f'(\xi)-k\big(\xi,f(\xi)\big)-\Big[-k\big(\xi,f(\xi)\big)+f'(\xi)\Big] k_\eta\big(\xi,f(\xi)\big)\cdot(\xi-\xi_0)\\
&=\Big[f'(\xi)-k\big(\xi,f(\xi)\big)\Big]\cdot\Big[1-k_\eta\big(\xi,f(\xi)\big)\cdot(\xi-\xi_0)\Big],
\end{split}\]
and then, from \eqref{der1} and the fact that $k_\eta$ is bounded, we have
\begin{equation}\label{F}
F'(\xi)\geq0,\qquad\forall \xi\in(\xi^*,\xi_0)
\end{equation}
provided that $\epsilon_3$ is sufficiently small.
Therefore, \eqref{F} implies that
\[0=F(\xi_0)\geq F(\xi^*)=\eta^*-f(\xi_0)-k(\xi^*,\eta^*)(\xi^*-\xi_0),
\]
which is a contradiction to \eqref{est_tr} by letting $\xi=\xi_0$ in \eqref{est_tr}. Hence, we have $P\in\overline\gamma_-$, which implies that
\((t,\xi,\eta,\zeta)\in Q_2,\) and this completes the proof of the proposition.

\end{proof}

\subsection{Proof of Theorem \ref{thm_global}}

Based on the results obtained in the above two subsections, we will give the proof of Theorem \ref{thm_global} in this subsection. Before it, the following lemmas and propositions are needed.

\begin{lemma}\label{lem_app} Let $W$ be the approximate solution constructed by
\eqref{pr_porous} and \eqref{pr_trans}. Then
\begin{equation}\label{est_max}
|W(t,\xi,\eta,\zeta)|\leq \widetilde{C}_1e^{\widetilde M_1 t}(1-\zeta),
\end{equation}
where
\[\widetilde C_1=\max\Big\{C_1,\|{W_1\over1-\zeta}\|_{L^\infty}\Big\},\quad
\widetilde M_1=\max\Big\{M_1,\|B-b\|_{L^\infty}\Big\}\]
with positive constants $C_1$ and $M_1$ being given in \eqref{def_para}.
Moreover, there exists $\tilde\beta$ depending only on $\|W_0\|_{L^\infty},\|W_1\|_{C^2}$ and the parameters of problem \eqref{pr_main}, such that
\begin{equation}\label{est_min}
W(t,\xi,\eta,\zeta)\geq \tilde\theta_0e^{-\tilde\beta t}\varphi,
\end{equation}
where $\varphi$ is given in \eqref{def_para} 
and $\tilde\theta_0=\min\Big\{{W_0\over\varphi},{W_1\over\varphi}\Big\}$.
\end{lemma}

\begin{proof}[\bf Proof.]
When $0\leq t\leq t_1$,  the estimates \eqref{est_max} and \eqref{est_min} follow from \eqref{est_por1} in Theorem \ref{thm_por} 
immediately. Assume that \eqref{est_max} holds for $0\leq t\leq t_{i}$ with $i\geq 1$, and  consider the case for $t_{i}\leq t\leq t_{i+1}$. If $i$ is even, from \eqref{est_por1}, 
we have
\[|W(t)|\leq \max\Big\{\|{W(t_{i})\over1-\zeta}\|_{L^\infty},\sqrt{\|{p_x\over U}\|_{L^\infty}}\Big\}e^{M_1(t-t_{i})}(1-\zeta)\leq \widetilde C_1e^{\widetilde M_1t}(1-\zeta),\]
by using the induction hypothesis.
If $i$ is odd, from \eqref{def_tr} and \eqref{def_b1} it follows
\[|W(t)|\leq |W(t_{i})|\cdot e^{\|B-b\|_{L^\infty}\cdot(t-t_{i})}\leq \widetilde C_1e^{\widetilde M_1t}(1-\zeta),\]
by using the induction hypothesis again, or
\[|W(t)|\leq W_1\cdot e^{\|B-b\|_{L^\infty}\cdot(t-t_{i})}\leq \widetilde C_1e^{\widetilde M_1t}(1-\zeta).\]
Thus, we conclude the estimate \eqref{est_max}.

Next, suppose that \eqref{est_min} holds for $0\leq t\leq t_{i}$ with $i\geq 1$. Then if $i$ is odd, from \eqref{est_por1} in Theorem \ref{thm_por},
we have that there exists $\tilde\beta$ depending on $\widetilde C_1, \|b\|_{L^\infty}$ and the parameters  in the problem \eqref{pr_main} such that
\[W\geq\min\Big\{{W(t_{i})\over\varphi}\Big\}e^{-\tilde\beta(t-t_{i})}
\varphi\geq\tilde\theta_0e^{-\tilde\beta t}\varphi.\]
If $i$ is odd, the estimate \eqref{est_min} is a direct consequence of the expression \eqref{exp_tr} in Proposition \ref{prop_tr}, provided that $\tilde\beta\geq\|B-b\|_{L^\infty}$.
Thus, we complete the proof.
\end{proof}

\begin{remark}
From Lemma \ref{lem_app} and by virtue of \eqref{est_b}, we find that there exists a constant $\widetilde C_0$, depending only on $\|W_0\|_{L^\infty},\|W_1\|_{C^2}$ and the parameters in the problem \eqref{pr_main}, and satisfying $\widetilde C_0\geq C_0$ with the constant $C_0$ being given in \eqref{ass_ib_glo}, such that
\begin{equation}\label{est_app}
\widetilde C_0^{-1}(1-\zeta)\leq W(t, \xi, \eta,\zeta)\leq \widetilde C_0(1-\zeta).
\end{equation}
\end{remark}

Now, we study the $L^1$ estimate of the first order derivatives of the approximate solution with respect to the spatial variables for obtaining the uniform estimate on the  total variation of  the solution. Before it, we give the following two propositions for the problem \eqref{pr_trans} of transport equation.

\begin{prop}\label{prop_tr1}
For the problem \eqref{pr_trans}, there exists a constant $C_4$ depending on the domain $D$,
the constant $\widetilde C_0$ given in \eqref{est_app}, $\|W_1\|_{C^1}$ and the $C^1$ estimates of the parameter in the problem \eqref{pr_trans}, such that for all $t\in[t_{i},t_{i+1}]$ and $\zeta\in(0,1)$,
\begin{equation}\label{est_tr-1}\begin{split}
&\int_D{|W_\xi(t)|+|W_\eta(t)|\over W^2(t)}(1-\zeta)^2d\xi d\eta
\leq \int_D{|W_\xi(t_{i})|+|W_\eta(t_{i})|\over W^2(t_{i})}(1-\zeta)^2
d\xi d\eta\\
&\qquad\qquad+C_4(t-t_{i})+C_4\int_{t_{i}}^t\int_D{|W_\xi(s)|+|W_\eta(s)|\over W^2(s)}(1-\zeta)^2d\xi d\eta ds.
\end{split}\end{equation}
\end{prop}

\begin{proof}[\bf Proof.]
For the problem \eqref{pr_trans}, we know that $({1\over W})_\xi$ satisfies
\begin{equation}\label{eq_tr-x}
\pd_t ({1\over W})_\xi+\zeta \pd_\xi \big[U({1\over W})_\xi\big]+\zeta kU\pd_\eta ({1\over W})_\xi+\zeta(kU)_\xi ({1\over W})_\eta-(B-b)({1\over W})_\xi=\frac{(B-b)_\xi}{W}.
\end{equation}
Taking \eqref{est_app} into account, we multiply the above equation \eqref{eq_tr-x} by $(1-\zeta)^2 sign W_\xi$ or $(1-\zeta)^2\cdot{W_\xi\over\sqrt{W_\xi^2}}$, and integrate the resulting equation over $D$ with respect to $(\xi,\eta)$, to obtain that
\begin{equation}\label{est_tr-x}\begin{split}
&\quad{d\over dt}\int_D{|W_\xi|\over W^2}(1-\zeta)^2d\xi d\eta+\int_D\zeta\pd_{\xi}\Big(U{|W_\xi|\over W^2}\Big)(1-\zeta)^2d\xi d\eta\\
&\qquad
+\int_D\zeta\pd_{\eta}\Big(kU{|W_\xi|\over W^2}\Big)(1-\zeta)^2d\xi d\eta-\int_D\zeta(kU)_\eta{|W_\xi|\over W^2}(1-\zeta)^2d\xi d\eta\\
&\leq \|(kU)_\xi\|_{L^\infty}\cdot\int_D{|W_\eta|\over W^2}(1-\zeta)^2d\xi d\eta+\|B-b\|_{L^\infty}\cdot\int_D{|W_\xi|\over W^2}(1-\zeta)^2d\xi d\eta\\
&\qquad+\int_D\big|(B-b)_\xi{(1-\zeta)^2\over W}\big|d\xi d\eta.
\end{split}\end{equation}
From  \eqref{pr_trans}, we obtain that on the boundary $\gamma_-$,
\[\zeta U(k_\tau\pd_\tau+k_n\pd_n)W|_{\gamma_-}=-\pd_t W_1-(B-b)|_{\gamma_-}\cdot W_1,\]
which implies
\begin{equation}\label{def_bd}
\zeta Uk_n\pd_n W|_{\gamma_-}= -\pd_t W_1-(B-b)|_{\gamma_-}\cdot W_1-\zeta U|_{\gamma_-}\cdot k_\tau\pd_\tau W_1~\triangleq ~b_2.\end{equation}
Obviously, there exist two bounded functions $a_1(\xi,\eta)$ and $a_2(\xi,\eta)$ defined on the boundary $\gamma_-$ such that  \[\pd_\xi=a_1\pd_n+a_2\pd_\tau,\quad{\rm on} ~\gamma_-.\]
Thus, from  \eqref{def_bd} one has
\begin{equation}\label{def_bd1}
\zeta Uk_n W_\xi|_{\gamma_-}= a_1b_2+a_2\zeta U|_{\gamma_-}\cdot k_n\pd_\tau W_1~\triangleq~b_3.
\end{equation}
Hence, it follows that by virtue of \eqref{def_bd1},
\begin{equation}\label{est_tr-x1}\begin{split}
&\quad\int_D\zeta\nabla_{(\xi, \eta)}\cdot\Big[U{|W_\xi|\over W^2}(1-\zeta)^2(1,k)\Big]d\xi d\eta=\int_{\pd D}\zeta U{|W_\xi|\over W^2}(1-\zeta)^2(1,k)\cdot\vec{n}dl\\
&\geq\int_{\gamma_-}\zeta Uk_n{|W_\xi|\over W^2}(1-\zeta)^2dl
=-\int_{\gamma_-}|b_3|\cdot{(1-\zeta)^2\over W_1^2}dl\\
&\geq-(C_0)^2\|b_3\|_{L^\infty}\cdot l(\gamma_-),
\end{split}\end{equation}
where $l(\gamma_-)$ is the length of $\gamma_-$ and the positive constant $C_0$ is given in \eqref{ass_ib_glo}.

By using \eqref{est_app}, it follows
\begin{equation}\label{est_tr-x2}
\int_D\big|(B-b)_\xi{(1-\zeta)^2\over W}|d\xi d\eta\leq \widetilde C_0\|(B-b)_\xi\|_{L^\infty}\cdot S(D),
\end{equation}
where $S(D)$ is the area of the domain $D$.

Plugging \eqref{est_tr-x1} and \eqref{est_tr-x2} into \eqref{est_tr-x}, we obtain that there exists a constant $\overline C_4$ depending on $D,\widetilde C_0,\|W_1\|_{C^1}$ and the $C^1$ estimates of the parameter in the problem \eqref{pr_trans}, such that
\begin{equation}\label{est_tr-x3}
{d\over dt}\int_D{|W_\xi(t)|\over W^2(t)}(1-\zeta)^2d\xi d\eta
\leq \overline C_4+\overline C_4\int_D{|W_\xi(t)|+|W_\eta(t)|\over W^2(t)}(1-\zeta)^2d\xi d\eta.
\end{equation}
Similarly, we can obtain another constant $\widetilde C_4$ such that
\begin{equation}\label{est_tr-y}
{d\over dt}\int_D{|W_\eta(t)|\over W^2(t)}(1-\zeta)^2d\xi d\eta
\leq \widetilde C_4+\widetilde C_4\int_D{|W_\xi(t)|+|W_\eta(t)|\over W^2(t)}(1-\zeta)^2d\xi d\eta.
\end{equation}
By letting $C_4=\overline C_4+\widetilde C_4$, we have that from \eqref{est_tr-x3} and \eqref{est_tr-y},
\begin{equation}\label{est_tr-x4}
{d\over dt}\int_D{|W_\xi(t)|+|W_\eta(t)|\over W^2(t)}(1-\zeta)^2d\xi d\eta
\leq  C_4+ C_4\int_D{|W_\xi(t)|+|W_\eta(t)|\over W^2(t)}(1-\zeta)^2d\xi d\eta.
\end{equation}
Then, integrating the above inequality \eqref{est_tr-x4} over $(t_{i},t)$ gives the estimate \eqref{est_tr-1} immediately,
and we complete the proof of the proposition.
\end{proof}

\begin{prop}\label{prop_tr2}
For the problem \eqref{pr_trans}, there exists a constant $ C_5$ depending on the domain $D$,
 the constant $\widetilde C_0$ given in \eqref{est_app}, $\|W_1\|_{C^1}$ and the $C^1$ estimates of the parameter in the problem \eqref{pr_trans}, such that for $t\in[t_{i},t_{i+1}]$ and $\zeta\in(0,1)$,
\begin{equation}\label{est_tr-z}\begin{split}
\int_D|W_\zeta(t)|d\xi d\eta
\leq &\int_D|W_\zeta(t_{i})|d\xi d\eta+C_5(t-t_{i})\\
&+C_5\int_{t_{i}}^t\int_D\Big(|W_\zeta(s)|
+{|W_\xi(s)|+|W_\eta(s)|\over W^2(s)}(1-\zeta)^2\Big)d\xi d\eta ds.
\end{split}\end{equation}

\end{prop}

\begin{proof}[\bf Proof.]
From the problem \eqref{pr_trans}, we know that $W_\zeta$ satisfies
\begin{equation}\label{eq_tr-z}
\pd_t W_\zeta+\zeta U(\pd_\xi+k\pd_\eta) W_\zeta+U(W_\xi+kW_\eta)+(B-b)W_\zeta=-(B-b)_\zeta W.
\end{equation}
Multiplying the above equation \eqref{eq_tr-z} by $sign W_\zeta$ or ${W_\zeta\over\sqrt{W_\zeta^2}}$, and integrating over $D$ with respect to $(\xi,\eta)$, it follows that
\begin{equation}\label{est_tr-z1}\begin{split}
&{d\over dt}\int_D|W_\zeta|d\xi d\eta+\int_D\zeta\Big(\pd_{\xi}\big(U|W_\zeta|\big)+\pd_{\eta}\big(kU|W_\zeta|\big)\Big)d\xi d\eta-\int_D \zeta\Big[U_\xi+(kU)_\eta\Big]\cdot|W_\zeta|d\xi d\eta\\
&\leq \|U\|_{L^\infty}\cdot\int_D|W_\xi|d\xi d\eta+\|kU\|_{L^\infty}\cdot\int_D|W_\eta|d\xi d\eta+\|B-b\|_{L^\infty}\cdot\int_D|W_\zeta|d\xi d\eta\\
&\quad+\|(B-b)_\zeta W\|_{L^\infty}\cdot S(D).
\end{split}\end{equation}
As in \eqref{est_tr-x1}, we have
\begin{equation}\label{est_tr-z2}\begin{split}
&\int_D\zeta\nabla_{(\xi,\eta)}\cdot\Big[U|W_\zeta|\cdot(1,k)\Big]d\xi d\eta \geq-\|k_nUW_{1,\zeta}\|_{L^\infty(\gamma_-)}\cdot l(\gamma_-).
\end{split}\end{equation}
Obviously, from the bounded estimate \eqref{est_app} for $W$ one has
\begin{equation}\label{est_tr-z3}
\int_D|W_\xi|d\xi d\eta\leq (\widetilde C_0)^2\int_D{|W_\xi|\over W^2}(1-\zeta)^2d\xi d\eta,
\end{equation}
and \begin{equation}\label{est_tr-z4}
\int_D|W_\eta|d\xi d\eta\leq (\widetilde C_0)^2\int_D{|W_\eta|\over W^2}(1-\zeta)^2d\xi d\eta.
\end{equation}
Plugging \eqref{est_tr-z2}, \eqref{est_tr-z3} and \eqref{est_tr-z4} into \eqref{est_tr-z1}, we obtain that there exists a constant $C_5$ depending on $D,\widetilde C_0,\|W_1\|_{C^1}$ and the $C^1$ estimates of the parameter in problem \eqref{pr_trans}, such that
\begin{equation}\label{est_tr-z5}
{d\over dt}\int_D|W_\zeta(t)|d\xi d\eta
\leq  C_5+ C_5\int_D\Big[|W_\zeta(t)|+{|W_\xi(t)|+|W_\eta(t)|\over W^2(t)}(1-\zeta)^2\Big]d\xi d\eta.
\end{equation}
Then, integrating the above inequality \eqref{est_tr-z5} over $(t_{i},t)$, the estimate \eqref{est_tr-z} in the proposition follows immediately.
\end{proof}

\begin{remark}\label{rem_tr}
Similar to the above proposition, one can show that $W_t$
satisfies an estimate similar to \eqref{est_tr-z}   in Proposition \ref{prop_tr2} by replacing the partial derivative in $\zeta$ by that in $t$.
\end{remark}

It is ready to give the $L^1$ estimate of the first order derivatives of the approximate solution $W$ constructed in \eqref{pr_porous}-\eqref{pr_trans} with respect to the spatial variables.

\begin{lemma}\label{lem_app1}
For any fixed $T>0$, let $W(t, \xi, \eta, \zeta)$ ($0\leq t\le T$) be the approximate solution of \eqref{pr_main} constructed in \eqref{pr_porous}-\eqref{pr_trans}.
Then, there exists a constant $M>0$, depending on $T$
and  the domain $D$,
the constant $\widetilde C_0$ given in \eqref{est_app}, $\|W_0\|_{L^\infty}$, $\|W_1\|_{C^3}$ and the $C^1$ estimates of the parameter in the problem \eqref{pr_main}, such that for all $t\in[0,T]$, we have
\begin{equation}\label{est_app1}\begin{split}
&\int_\Omega\Big[|W_\zeta|+|W_\xi|+|W_\eta|\Big](t, \cdot)d\xi d\eta d\zeta\\
&\leq M\Big(1+e^{M t}\cdot\int_\Omega(|W_{0,\zeta}|+|W_{0,\xi}|+|W_{0,\eta}|)d\xi d\eta d\zeta\Big).
\end{split}\end{equation}
\end{lemma}

\begin{proof}[\bf Proof.]
The proof is divided into the following three steps.

(1) When $t\in(t_{i},t_{i+1}]$ for even
$i$, $W$ is determined by the initial boundary value problem \eqref{pr_porous} for a porous medium-type equation. From Theorem \ref{thm_por}
 and the boundedness \eqref{est_app} of $W,$ we obtain that for $t\in(t_{i},t_{i+1}]$,
\begin{equation}\label{est_app2}
\int_\Omega|W_\zeta|(t, \cdot)d\xi d\eta d\zeta\leq\int_\Omega|W_\zeta|(t_{i},\cdot)d\xi d\eta d\zeta+\int_D\Big[W(t,\xi,\eta,0)-W(t_{i},\xi,\eta,0)\Big]d\xi d\eta.
\end{equation}
Moreover, there exists a constant $\widetilde C_3$ depending on the domain $D$, the constant $\widetilde C_0$ given in \eqref{est_app}, $\|W_1\|_{C^3}$ and the $C^1$ estimates of the parameter in problem \eqref{pr_main}, such that
\begin{equation*}
\begin{split}
&\int_\Omega\Big[{|W_\xi|+|W_\eta|\over W^2}(1-\zeta)^2\Big](t, \cdot)d\xi d\eta d\zeta\\
&\leq
\int_\Omega\Big[{|W_\xi|+|W_\eta|\over W^2}(1-\zeta)^2\Big](t_{i},\cdot)d\xi d\eta d\zeta
+\widetilde C_3\Big(1+\int_\Omega|W_\zeta|(t_{i},\cdot)d\xi d\eta d\zeta\Big)\cdot(t-t_{i})\\
&\quad
+\widetilde C_3\int_{t_{i}}^t\int_\Omega\Big[{|W_\xi|+|W_\eta|\over W^2}(1-\zeta)^2\Big](s,\cdot)d\xi d\eta d\zeta ds.
\end{split}\end{equation*}
which implies that by using the Gronwall inequality,
\begin{equation}\label{est_app3}\begin{split}
&\int_\Omega\Big[{|W_\xi|+|W_\eta|\over W^2}(1-\zeta)^2\Big](t,\cdot)d\xi d\eta d\zeta\\
&\leq
e^{\widetilde C_3(t-t_{i})}\cdot\int_\Omega\Big[{|W_\xi|+|W_\eta|\over W^2}(1-\zeta)^2\Big](t_{i},\cdot)d\xi d\eta d\zeta\\
&\quad+(e^{\widetilde C_3(t-t_{i})}-1)\cdot\Big(1+\int_\Omega|W_\zeta|(t_{i},\cdot)d\xi d\eta d\zeta\Big).
\end{split}\end{equation}
 Combining \eqref{est_app2} with \eqref{est_app3}, it follows that
\begin{equation}\label{est_app4}\begin{split}
&\int_\Omega\Big[|W_\zeta|+{|W_\xi|+|W_\eta|\over W^2}(1-\zeta)^2\Big](t,\cdot)d\xi d\eta d\zeta\\
&\leq
e^{\widetilde C_3(t-t_{i})}\cdot\int_\Omega\Big[|W_\zeta|+{|W_\xi|+|W_\eta|\over W^2}(1-\zeta)^2\Big](t_{i},\cdot)d\xi d\eta d\zeta\\
&\quad+e^{\widetilde C_3(t-t_{i})}-1+\int_D\Big[W(t,\cdot)-W(t_{i},\cdot)\Big]\Big|_{\zeta=0}d\xi d\eta.
\end{split}\end{equation}

(2) When $t\in(t_{i},t_{i+1}]$ for odd $i$, we obtain
$W$ by the problem \eqref{pr_trans} for a transport equation. From Propositions \ref{prop_tr1}, \ref{prop_tr2} and the estimate \eqref{est_app}, it follows that there exists a constant $C_6$, depending on the domain $D$, the constant $\widetilde C_0$ given in \eqref{est_app}, $\|W_1\|_{C^3}$ and the $C^1$ estimates of the parameter in problem \eqref{pr_main}, such that for $t\in(t_{i},t_{i+1}]$,
\begin{equation*}
\begin{split}
&\int_\Omega\Big[|W_\zeta|+{|W_\xi|+|W_\eta|\over W^2}(1-\zeta)^2\Big](t,\cdot)d\xi d\eta d\zeta\\
&\leq
\int_\Omega\Big[|W_\zeta|+{|W_\xi|+|W_\eta|\over W^2}(1-\zeta)^2\Big](t_{i},\cdot)d\xi d\eta d\zeta+C_6(t-t_{i})\\
&\quad+C_6\int_{t_{i}}^t\int_\Omega\Big[|W_\zeta|+{|W_\xi|+|W_\eta|\over W^2}(1-\zeta)^2\Big](s, \cdot)d\xi d\eta d\zeta ds,
\end{split}\end{equation*}
which implies that by using the Gronwall inequality,
\begin{equation}\label{est_app6}\begin{split}
&\int_\Omega\Big[|W_\zeta|+{|W_\xi|+|W_\eta|\over W^2}(1-\zeta)^2\Big](t,\cdot)d\xi d\eta d\zeta\\
&\leq
e^{C_6(t-t_{i})}\cdot\int_\Omega\Big[|W_\zeta|+{|W_\xi|+|W_\eta|\over W^2}(1-\zeta)^2\Big](t_{i},\cdot)d\xi d\eta d\zeta+e^{C_6(t-t_{i})}-1.
\end{split}\end{equation}

(3) On the other hand, when $i$ is odd we have that from Proposition \ref{prop_tr},
\begin{equation}\label{def_z0}
W(t_{i+1},\xi,\eta,0)=W(t_{i},\xi,\eta,0)\cdot\exp
\Big\{-\int_{t_{i}}^{t_{i+1}}(B-b)(s,\xi,\eta,0)ds\Big\}.
\end{equation}

By combining \eqref{est_app4} , \eqref{est_app6} and \eqref{def_z0}, and letting $C_7=\max\{\widetilde C_3,C_6\}$, we obtain that for any $t\in(t_{i},t_{i+1}]$,
\begin{equation}\label{est_app7}
\begin{split}
&\int_\Omega\Big[|W_\zeta|+{|W_\xi|+|W_\eta|\over W^2}(1-\zeta)^2\Big](t,\cdot)d\xi d\eta d\zeta\\
&\leq
e^{C_7(t-t_{i})}\cdot\int_\Omega\Big[|W_\zeta|+{|W_\xi|+|W_\eta|\over W^2}(1-\zeta)^2\Big](t_{i},\cdot)d\xi d\eta d\zeta
\\&\quad
+e^{C_7(t-t_{i})}-1+G_{i}(t),
\end{split}
\end{equation}
where
\[\begin{split}G_{i}(t)&=\int_D\Big[W(t,\cdot)-W(t_{i},\cdot)\Big]\Big|_{\zeta=0}d\xi d\eta\\
&=\int_D\Big[W(t,\cdot)-W(t_{i-1},\cdot)\cdot\exp
\Big\{-\int_{t_{i-1}}^{t_i}(B-b)(s,\cdot)ds\Big\}\Big]\Big|_{\zeta=0}d\xi d\eta
\end{split}\]
for even $i$, and $G_{i}(t)=0$ for odd $i$. Hence,
\begin{equation}\label{est_G}
G_{i}(t)\leq 2\|W\|_{L^\infty}\cdot S(D),\qquad ~\forall t\in[0,T].\end{equation}

Therefore, \eqref{est_app7} implies that
\begin{equation}\label{est_app8}\begin{split}
&\int_\Omega\Big[|W_\zeta|+{|W_\xi|+|W_\eta|\over W^2}(1-\zeta)^2\Big](t,\cdot)d\xi d\eta d\zeta\\
&\leq
e^{C_7t}\cdot\int_\Omega\Big[|W_{0,\zeta}|+{|W_{0,\xi}|+|W_{0,\eta}|\over W_0^2}(1-\zeta)^2\Big]d\xi d\eta d\zeta+e^{C_7t}-1+G_{i}(t)+F_{i}(t),
\end{split}\end{equation}
where
\[\begin{split}
&F_{i}(t)=\sum_{j=0}^{i-1}
e^{C_7(t-t_{j+1})}G_j(t_{j+1})\\
&=\sum_{0\leq j\leq i-2,~j:odd}e^{C_7(t-t_{j+1})}\cdot
\int_D\Big[W(t_{j+1},\cdot)-W(t_{j-1},\cdot)\cdot\exp
\Big\{-\int_{t_{j-1}}^{t_j}(B-b)(s,\cdot)ds\Big\}\Big]\Big|_{\zeta=0}d\xi d\eta.
\end{split}\]

Note that $W_0|_{\zeta=0}>0$, then
\begin{equation}\label{est_F}\begin{split}
&\qquad F_{i}(t)\\
&\leq e^{2C_7{T\over n}}\Big\|W|_{\zeta=0}\Big\|_{L^\infty}\cdot S(D)-e^{C_7(t-t_2)}\int_D\Big[W_0(\cdot)\cdot
\exp\Big\{-\int_0^{t_1}(B-b)(s,\cdot)ds\Big\}\Big]\Big|_{\zeta=0}d\xi d\eta\\
&\quad+\sum_{k=1}^{[{i\over2}]-1}\int_De^{C_7(t-t_{2k})}\Big[W(t_{2k},\cdot)\cdot
\Big(
1-\exp\{-2C_7{T\over n}-\int_{t_{2k}}^{t_{2k+1}}(B-b)(s,\cdot)ds\}
\Big)\Big]\Big|_{\zeta=0}d\xi d\eta\\
&\leq \Big\|W|_{\zeta=0}\Big\|_{L^\infty}\cdot S(D)\cdot(e^{2C_7{T\over n}}+I),
\end{split}\end{equation}
where $[{i\over2}]$ is the largest integer less than or equal to ${i\over2}$, and
\[I=\sum_{k=1}^{[{i\over2}]-1}e^{C_7(t-t_{2k})}
\Big\|1-\exp\{-\int_{t_{2k}}^{t_{2k+1}}\Big[2C_7+(B-b)(s,\xi,\eta,0)\Big]ds\}\Big\|_{L^\infty(D)}.\]
By choosing a constant $C_7$ satisfying that $2C_7\geq\Big\|(B-b)|_{\zeta=0}\Big\|_{L^\infty}$, we obtain
\begin{equation}\label{est_I}\begin{split}
I&\leq \sum_{k=1}^{[{i\over2}]-1}e^{C_7{T\over n}(i-2k)}\cdot
\|\int_{t_{2k}}^{t_{2k+1}}\Big[2C_7+(B-b)(s,\xi,\eta,0)\Big]ds\|_{L^\infty(D)}\\
&\leq \sum_{k=1}^{[{i\over2}]-1}e^{C_7{T\over n}(i-2k)}\cdot4C_7{T\over n}\leq 4C_7t e^{C_7t}.
\end{split}\end{equation}
Plugging \eqref{est_G}, \eqref{est_F} and \eqref{est_I} into \eqref{est_app8}, it follows that there exists a positive constant $C_8$, depending on $T,S(D),\|W\|_{L^\infty}$ and $C_7$, such that
\begin{equation*}
\begin{split}
&\int_D\Big[|W_\zeta|+{|W_\xi|+|W_\eta|\over W^2}(1-\zeta)^2\Big](t,\cdot)d\xi d\eta\\
&\leq e^{C_7t}\int_D\Big[|W_{0,\zeta}|+{|W_{0,\xi}|+|W_{0,\eta}|\over W_0^2}(1-\zeta)^2\Big]d\xi d\eta+C_8,
\end{split}\end{equation*}
from which the estimate \eqref{est_app1} follows immediately by using \eqref{est_app}. Thus, we complete the proof of the lemma.
\end{proof}

We are now ready to give the proof of the existence of weak solution
as follows.

\begin{proof}[\bf Proof of Theorem \ref{thm_global}.] Denote by $W^n(t,\xi,\eta,\zeta)$  the approximate solution constructed in \eqref{pr_porous}-\eqref{pr_trans}.
Let $Y$ be the dual space of $H_0^2(\Omega)$. First, we claim that
\begin{equation}\label{est_H2}
\Big\|\pd_t\Big({(1-\zeta)^2\over W^n}\Big)\Big\|_{L^2(0,T;Y)}\leq C_9
\end{equation}
with a constant $C_9$ independent of $n$. Indeed, if $i$ is even, we have that from \eqref{pr_porous},
\[\begin{split}
&\int_{t_{i}}^{t_{i+1}}\Big\|\pd_t\Big({(1-\zeta)^2\over W^n}\Big)(t,\cdot)\Big\|^2_Ydt=4\int_{t_{i}}^{t_{i+1}}
\Big\|(1-\zeta)^2\Big[\pd_\zeta^2 W^n+A
\pd_\zeta(\frac{1}{W^n})-{b\over W^n}\Big](t,\cdot)\Big\|^2_Ydt\\
&=4\int_{t_{i}}^{t_{i+1}}\Big(\sup_{\|\psi\|_{H_0^2(\Omega)}\leq1}
\int_\Omega\Big[\pd_\zeta^2 W^n+A
\pd_\zeta
(\frac{1}{W^n})-{b\over W^n}\Big](1-\zeta)^2\cdot \psi d\xi d\eta d\zeta\Big)^2dt\\
&\leq C_{10}\int_{t_{i}}^{t_{i+1}}\Big[\|W^n(t,\cdot)\|_{L^\infty}+\|{1-\zeta\over W^n(t,\cdot)}\|_{L^\infty}\Big]dt,
\end{split}\]
with a uniform constant $C_{10}$ depending only on $D$, $\|A\|_{C^1}$ and $\|b\|_{L^\infty}$.

If $i$ is odd, it follows that by using \eqref{pr_trans},
\[\begin{split}
&\int_{t_{i}}^{t_{i+1}}\Big\|\pd_t \Big({(1-\zeta)^2\over W^n}\Big)(t,\cdot)\Big\|^2_Ydt=4\int_{t_{i}}^{t_{i+1}}
\Big\|(1-\zeta)^2\Big[-\zeta U(\pd_\xi+k
\pd_\eta)
 ({1\over W^n})+{B-b\over W^n}\Big](t,\cdot)\Big\|^2_Ydt\\
&=4\int_{t_{i}}^{t_{i+1}}\Big(\sup_{\|\psi\|_{H_0^2(\Omega)}\leq1}
\int_\Omega\Big[-\zeta U(\pd_\xi+k
\pd_\eta)
 ({1\over W^n})+{B-b\over W^n}\Big](1-\zeta)^2\cdot \psi d\xi d\eta d\zeta\Big)^2dt\\
&\leq C_{11}\int_{t_{i}}^{t_{i+1}}\|{1-\zeta\over W^n(t,\cdot)}\|_{L^\infty}dt,
\end{split}\]
with a uniform constant $C_{11}$ depending only in $D$, $\|U\|_{C^1}$ and $\|B-b\|_{L^\infty}$.
Thus, by using \eqref{est_app},  the estimate \eqref{est_H2} follows immediately.

Next, from \eqref{est_app} and Lemma \ref{lem_app1}, we know that
\[\Big\|{(1-\zeta)^2\over W^n}\Big\|_{L^\infty(0,T;W^{1,1}(\Omega))}\leq C_{12},\]
where $C_{12}$ is a positive constant independent of $n$. Hence, by using the Lions-Aubin Lemma (see \cite{BS} for instance),
we conclude that $\Big\{{(1-\zeta)^2\over W^n}\Big\}$ is compact in $L^2\big((0,T)\times\Omega\big)$. Therefore, we may assume that
\[{(1-\zeta)^2\over W^n}\rightarrow {(1-\zeta)^2\over W},\quad in~L^2\big((0,T)\times\Omega\big),\]
and then
\[W^n\rightarrow W,\quad~{\rm a.e.}~in ~(0,T)\times\Omega.\]
In particular,
\[W^n\rightarrow W,\quad in~L^2\big((0,T)\times\Omega\big).\]
Thus, for any $\psi\in C_0^\infty\big((0,T)\times\Omega\big)$, we have that
\begin{equation}\label{lim}\begin{split}
&\int_0^T\int_\Omega\Big[({1\over W^n})_t+\zeta U({1\over W^n})_\xi+\zeta Uk({1\over W^n})_\eta-{B-b\over W^n}\Big]\psi d\xi d\eta d\zeta ds\\
&=2\sum_{i:odd}\int_{t_{i}}^{t_{i+1}}\int_\Omega\Big[-\pd_\zeta^2 W^n-A({1\over W^n})_\zeta+\zeta U({1\over W^n})_\xi+\zeta Uk({1\over W^n})_\eta
+{2b-B\over W^n}\Big]\psi d\xi d\eta d\zeta ds.
\end{split}\end{equation}
From Theorem \ref{thm_por} and Remark \ref{rem_tr},
and by integrating by parts in the above equality \eqref{lim} it yields that,
\begin{equation}\label{lim1}\begin{split}
&\int_0^T\int_\Omega{1\over W^n}\Big[\psi_t+\zeta (U\psi)_\xi+\zeta (Uk\psi)_\eta-(B-b)\psi\Big]d\xi d\eta d\zeta ds\\
&=2\sum_{i:odd}\int_{t_{i-1}}^{t_i}\int_\Omega\Big[ W^n \psi_{\zeta\zeta}+{1\over W^n}\Big(-(A\psi)_\zeta+\zeta (U\psi)_\xi+\zeta (Uk\psi)_\eta
+(2b-B)\psi\Big)\Big] d\xi d\eta d\zeta ds.
\end{split}\end{equation}
Letting $n\rightarrow \infty$ in \eqref{lim1} and noting that $t_i-t_{i-1}={T\over n}$, we have
\begin{equation*}
\begin{split}
&\int_0^T\int_\Omega{1\over W}\Big[\psi_t+\zeta (U\psi)_\xi+\zeta (Uk\psi)_\eta-(B-b)\psi\Big]d\xi d\eta d\zeta ds\\
&=\int_{0}^{T}\int_\Omega\Big[ W \psi_{\zeta\zeta}+{1\over W}\Big(-(A\psi)_\zeta+\zeta (U\psi)_\xi+\zeta (Uk\psi)_\eta
+(2b-B)\psi\Big)\Big] d\xi d\eta d\zeta ds,
\end{split}\end{equation*}
which implies
\begin{equation}\label{lim3}\begin{split}
&\int_0^T\int_\Omega\Big\{{1\over W}\Big[\psi_t+\zeta (U\psi)_\xi+\zeta (Uk\psi)_\eta+(A\psi)_\zeta-B\psi\Big]-W\psi_{\zeta\zeta}\Big\}d\xi d\eta d\zeta ds=0.
\end{split}\end{equation}
Therefore, from \eqref{lim3} we know that the function $W$ satisfies the equation of problem \eqref{pr_main} in the sense of distribution. Moreover, we can obtain that $W$ satisfies the estimate \eqref{est_app1} by letting $n\rightarrow\infty$, which implies that
\[W\in L^\infty\big(0,T;BV(\Omega)\big).\]

It remains to verify that $W$ satisfies the initial and  boundary conditions given in
\eqref{pr_main}.
It follows from Lemma \ref{lem_app} that,
\[\lim_{\zeta\rightarrow1}W(t,\xi,\eta,\zeta)=0,\quad {\rm a.e.}~in~(0,T)\times D.\]
We can verify the other boundary conditions in \eqref{pr_main} for $W$ in the sense of distribution, respectively,
through similar process as above to show that $W$ satisfies the equation of \eqref{pr_main} in the sense of distribution. For example, we have the following equality holds
\begin{equation*}
\begin{split}
&\int_0^T\int_\Omega{1\over W}\Big[\psi_t+\zeta (U\psi)_\xi+\zeta (Uk\psi)_\eta-(B-b)\psi\Big]d\xi d\eta d\zeta ds\\
&={1\over2}\int_{0}^{T}\int_\Omega\Big[ W \psi_{\zeta\zeta}+{1\over W}\Big(-(A\psi)_\zeta+\zeta (U\psi)_\xi+\zeta (Uk\psi)_\eta\\
&\qquad\qquad\qquad+(2b-B)\psi\Big)\Big] d\xi d\eta d\zeta ds-{1\over2}\int_0^T\int_D{p_x\over U}\cdot {\psi\over W}|_{\zeta=0}d\xi d\eta ds,
\end{split}\end{equation*}
for any
\[\psi\in C_0^\infty\big((0,T)\times D\times(-1,1)\big)\]
with $\psi_\zeta|_{\zeta=0}=0$. Therefore, we complete the proof of Theorem \ref{thm_global}.

\end{proof}

\vspace{.15in}

{\bf Acknowledgements:}
The first two authors' research was supported in part by
National Natural Science Foundation of China (NNSFC) under Grant No. 91230102, and the second author's research was also supported by Shanghai Committee of Science and Technology under Grant No. 15XD1502300. The last author's research was supported by the General Research Fund of Hong Kong,
CityU No. 103713.

\end{document}